\documentclass[dvipdfmx]{amsart}

\usepackage{amsmath}
\usepackage{amssymb}
\usepackage{xcolor}
\usepackage{graphicx}
\usepackage[T1]{fontenc}
\usepackage{braket}

\usepackage{amsthm}
\theoremstyle{definition}
\newtheorem{definition}{Definition}[section]
\newtheorem{proposition}[definition]{Proposition}
\newtheorem{theorem}[definition]{Theorem}
\newtheorem{corollary}[definition]{Corollary}
\newtheorem{lemma}[definition]{Lemma}
\newtheorem{example}[definition]{Example}
\newtheorem{remark}[definition]{Remark}

\newcommand{\RR}{\mathbb{R}}
\newcommand{\ZZ}{\mathbb{Z}}
\newcommand{\opr}{\triangleleft}
\newcommand{\img}{\mathop{\mathrm{im}}}
\newcommand{\ord}[1]{\lvert #1 \rvert}
\newcommand{\norm}[1]{\lVert #1 \rVert}
\newcommand{\bigast}{\mathop{\scalebox{1.5}{\raisebox{-0.2ex}{$\ast$}}}}
\newcommand{\adg}{\mathop{\mathrm{Ad}}}
\newcommand{\conj}{\mathop{\mathrm{Conj}}}
\newcommand{\homog}{\mathrm{h}}
\newcommand{\bdd}{\mathrm{b}}
\renewcommand{\hom}{\mathop{\mathrm{Hom}}}

\usepackage{hyperref}
\hypersetup{colorlinks=true}

\begin{document}
    \title{Quasimorphisms of free products of racks and quandles}
    \author{Masamitsu Aoki}
    \thanks{aokim@math.sci.hokudai.ac.jp}
    \date{\today}

    \begin{abstract}
        We show that the second bounded cohomology of the free product of racks and quandles is infinite-dimensional as a real vector space. 
        This is similar to the case of groups.
        As a corollary, we show that the second bounded cohomology of the free rack and the free quandle is infinite-dimensional.
        We also give another proof of this corollary using homogeneous group quasimorphisms.
    \end{abstract}

    \maketitle

    \section{Introduction}
    A rack (see, for example \cite{Fenn1992-dd}) and a quandle (introduced independently in \cite{Joyce1982-my,Matveev1982-ln}) are the algebraic objects whose axioms reflect the Reidemeister moves or the corresponding properties of conjugation in a group. 
    As in the case of groups, we can consider the cohomology of racks and quandles (see for example \cite{Nosaka2017-ms} and references therein). 
    The rack and quandle cocycles provide knot invariants. 
    In \cite{Kedra2024-jg}, K\k{e}dra introduced the bounded version of rack and quandle cohomology.

    The study of bounded cohomology of groups started to develop from Gromov's pioneering paper \cite{Gromov1982-vs}.
    Ordinary cohomology and bounded cohomology behave differently. 
    For example, on the free group, the second cohomology is trivial while the second bounded cohomology is not.
    More precisely, the second bounded cohomology of the free group of rank $\geq 2$ in trivial real coefficients is infinite-dimensional as a real vector space.
    It is known that the second bounded cohomology is infinite-dimensional for some classes of groups, including
    free groups and surface groups \cite{Brooks1981-ez,Mitsumatsu1984-dt}, 
    non-elementary hyperbolic groups \cite{Epstein1997-od}, and
    relatively hyperbolic groups \cite{Bestvina2002-fg}.

    Similar to the case of free groups, the second and higher cohomology of the free racks and the free quandles are trivial \cite{Farinati2014-zt}.
    (Note that their cochain complex begins at degree $-1$ in our convention. 
    Thus, what they refer to as 'first' cohomology corresponds to our second cohomology.)
    It follows from observations in \cite{Kedra2024-jg} that the second bounded cohomology of the free product of racks and quandles (especially, the free racks and the free quandles) is non-trivial.
    Saraf and Singh \cite{Saraf2025-wo} investigate the relation of bounded quandle cohomology with the stable commutator lengths and amenability, provide sufficient conditions for the infinite-dimensionality of the second bounded cohomology of link quandles and the free products of quandles, and prove that the second bounded cohomology of the free quandle is infinite-dimensional. 
    As our main result, we will show the infinite-dimensionality of the second bounded cohomology of the free products of racks by taking a different approach:
    \begin{theorem}\label{theorem:main}
        Let $X = \bigast_{s \in S} X_s$ be the free product of finitely generated racks with $2 \leq \ord{S} < \infty$.
        Then the second bounded cohomology $H_\bdd^2 (X, \RR)$ is infinite-dimensional as a real vector space.
    \end{theorem}\noindent
    Since the free rack on a set $S$ is the free product of trivial racks $\{s\}$ ($s \in S$), we immediately obtain the following corollary:
    \begin{corollary}\label{corollary:main}
        Let $X$ be the free rack of finite rank $\geq 2$. 
        Then $H_\bdd^2 (X, \RR)$ is infinite-dimensional.
    \end{corollary}\noindent
    This is another proof of Corollary 5.7 in \cite{Saraf2025-wo}.

    These results hold also in case of quandles almost in parallel.
    Hereafter, we only mention racks and omit the phrase `and/or \dots quandle(s)' unless an additional explanation is needed.

    We will use quasimorphisms on racks to prove our results in Section \ref{section:free-product}.
    These are also introduced by K\k{e}dra in \cite{Kedra2024-jg}.
    Quasimorphisms on groups are used to prove infinite-dimensionality in the above studies.
    This approach can be applied similarly to the case of racks.
    We propose the method to construct a quasimorphism on the free product of racks using a quasimorphism on the free product of groups which was introduced by Rolli in \cite{Rolli2009-wy}.
    This construction provides an injective linear map from an infinite-dimensional vector space to the second bounded cohomology.

    We will give another proof of Corollary \ref{corollary:main} in Section \ref{section:homogeneous}.
    Similar argument also works when we use a homogeneous group quasimorphism instead of the quasimorphism of Rolli.
    We obtain the injective linear map from the space of homogeneous quasimorphism on the adjoint group of the free product $X$ of racks to the second bounded cohomology of $X$.
    The space of homogeneous quasimorphisms on $\Gamma$ is also infinite-dimensional if $\Gamma$ is the adjoint group of the free racks.
    Indeed, this follows from the facts that the adjoint group of the free rack is the free group, that the free group (of rank $\geq 2$) is non-elementary word-hyperbolic, and that the second bounded cohomology of any non-elementary word-hyperbolic group $\Gamma$ is infinite-dimensional.

    \subsubsection*{Acknowledgements}
    The author would like to thank Jarek K\k{e}dra, Morimichi Kawasaki and Mitsuaki Kimura for advice on this work, and the anonymous referees for helpful comments.

    \section{Preliminaries}\label{section:preliminaries}
    We start with recalling the definition and examples of racks.
    The free product of racks and quandles is our main object.
    Then we recall the (bounded) cohomology of racks and quandles which is introduced by K\k{e}dra in \cite{Kedra2024-jg}.
    For further details, we refer to \cite{Fenn1992-dd} for racks, to \cite{Nosaka2017-ms} for quandles including their cohomology, and to \cite{Calegari2009-yk,Frigerio2017-kr} for the basics of bounded cohomology and quasimorphisms of groups.

    As we mentioned in Introduction, we will give statements on racks and provide additional explanations for quandles only if needed.

    \subsection{Racks and quandles}
    A \emph{rack} is a set $X$ together with a binary operation $\opr \colon X \times X \to X$ satisfying the following axioms:
    \begin{enumerate}
        \item the \emph{rack identity}: $(x \opr y) \opr z = (x \opr z) \opr (y \opr z)$ for any $x, y, z \in X$, and
        \item a map $\psi_y \colon X \to X$ defined by $\psi_y (x) = x \opr y$ is bijective for any $y \in X$.
    \end{enumerate}
    A \emph{quandle} is a rack $(X, \opr)$ satisfying 
    \begin{enumerate}\setcounter{enumi}{2}
        \item $x \opr x = x$ for any $x \in X$. \label{axiom:idempotence}
    \end{enumerate}
    We write $\psi_y^n (x) = x \opr^n y$ for any $n \in \ZZ$. 
    
    A \emph{rack homomorphism} between racks is a map $f \colon X \to Y$ with $f (x \opr y) = f (x) \opr f (y)$.
    In the case of quandles, it is called a \emph{quandle homomorphism}.

    \begin{example}
        A set $X$ together with the operation $x \opr y = x$ is a rack, called the \emph{trivial rack}. 
        When we consider it as a quandle, we call it the \emph{trivial quandle}.
    \end{example}
    \begin{example}
        Let $G$ be a group.
        A set $G$ together with the operation $g \opr h = h^{-1} g h$ is a rack, called the \emph{conjugacy rack} of $G$.
        As a quandle, we call it the \emph{conjugacy quandle}.
    \end{example}
    
    The \emph{adjoint group} of $X$ is a group $\adg (X)$ presented as
    \begin{equation}
        \left\langle e_x \ (x \in X) \mid e_x e_y = e_y e_{x \opr y} \right\rangle.
    \end{equation}
    Since a rack homomorphism $f \colon X \to Y$ induces a group homomorphism $f_\sharp \colon \adg (X) \to \adg (Y)$, this gives rise to the functor $\adg (-)$ from the category of racks to the category of groups.
    This is left adjoint to the functor $\conj (-)$ which makes a group into the conjugacy rack.
    
    The adjoint group $\adg (X)$ acts on $X$ by $x \cdot e_y = x \opr y$.
    A \emph{connected component} of $X$ is an orbit of this action.

    \subsection{Free products}
    As in the case of groups, we have a free rack and a free product of racks. 
    Although they can be characterized as colimits in the category of racks by universal property (see \cite{Farinati2014-zt} or \cite{Fenn1992-dd}), here we give explicit constructions.
    \begin{example}[\cite{Farinati2014-zt, Fenn1992-dd}]\label{example:free_rack}
        Let $S$ be a set and $F^G (S)$ the free group on $S$.
        The \emph{free rack} on $S$ is a set $F^R (S) = S \times F^G (S)$ together with the operation
        \begin{equation}
            (s, g) \opr (t, h) = (s, g h^{-1} t h).
        \end{equation}
    \end{example}
    \begin{example}[\cite{Farinati2014-zt,Joyce1982-my}]
        The \emph{free quandle} on $S$ is the union of the conjugacy classes of elements of $S$ inside $F^G (S)$, $F^Q (S) = \bigcup_{s \in S} \{ g^{-1} s g \in F^G (S) \mid g \in F^G (S) \}$, with the operation $\alpha \opr \beta = \beta^{-1} \alpha \beta$.
    \end{example}
    \begin{example}[\cite{Fenn1992-dd}]
        Let $S$ be a set. 
        For a family of racks $X_s$ ($s \in S$), the \emph{free product} is a rack $\bigast_{s \in S} X_s$ consists of elements of the form $(x, g)$ where $x \in X_t$ for some $t \in S$ and $g \in \bigast_{s \in S} \adg (X_s)$ under the equivalence generated by 
        \begin{equation}\label{equation:equivalence-free-product}
            (x, g w) \sim (x \cdot g, w)
        \end{equation}
        where $x \in X_t$, $g \in \adg (X_t)$ for each $t \in S$, and $w \in \bigast_{s \in S} \adg (X_s)$.
        The operation is 
        \begin{equation}\label{equation:operation-free-product}
            (x, g) \opr (y, h) = (x, g h^{-1} e_y h).
        \end{equation}
        This is well-defined by the definition of the adjoint group.
        If $X_s$ are quandles, the free product is also a quandle. 
    \end{example}
    Since the adjoint group of a trivial rack of the one element set $\{\ast\}$ is $\ZZ$, the free rack $F^R (S)$ is the free product of the trivial racks $\{ s \}$ ($s \in S$).
    Since a left adjoint functor preserves colimits, the adjoint group of the free product $\bigast_{s \in S} X_s$ is the free product of the adjoint groups, $\bigast_{s \in S} \adg (X_s)$. 
    Especially, the adjoint group of the free rack $F^R (S)$ is the free group $F^G (S)$.

    \begin{remark}
        The quotient $F^R (S) / {\sim}$ under the equivalence generated by $(s, g) \sim (s, s g)$ for any $s \in S$ and $g \in F^G (S)$ is a quandle.
        This quotient can be identified with the free quandle $F^Q (S)$ defined above by associating a conjugation $g^{-1} s g$ with the class represented by $(s, g)$. 
        The above equivalence is descended from the equivalence \eqref{equation:equivalence-free-product}.
        Thus, we can say that the free quandle on $S$ is the free product of the trivial quandles $\{s\}$ as in the case of racks.
    \end{remark}

    A rack $X$ is \emph{finitely generated} if there is a finite subset $S \subset X$ such that, for any $x \in X$, there are an integer $n \geq 0$, elements $s_0, s_1, \dots, s_n \in S$, and $\varepsilon_1, \dots, \varepsilon_n \in \{ \pm 1 \}$ such that 
    \begin{equation}
        x = s_0 \opr^{\varepsilon_1} s_1 \opr \cdots \opr^{\varepsilon_n} s_n.
    \end{equation}
    By definition, the adjoint group $\adg (X)$ is generated by $\{ e_x \mid x \in X \}$.
    If a rack $X$ is finitely generated by $S$, each generator $e_x$ can be written in the form 
    \begin{equation}
        e_x = e_{s_n}^{- \varepsilon_n} \cdots e_{s_1}^{- \varepsilon_1} \cdot e_{s_0} \cdot e_{s_1}^{\varepsilon_1} \cdots e_{s_n}^{\varepsilon_n}
    \end{equation}
    for some $s_0, s_1, \dots, s_n \in S$ and $\varepsilon_1, \dots, \varepsilon_n \in \{ \pm 1 \}$.
    Thus, we obtain the following:
    \begin{lemma}
        If a rack $X$ is finitely generated by $S$, then the adjoint group $\adg (X)$ is finitely generated by $\{ e_s \mid s \in S \}$.
    \end{lemma}

    \subsection{Bounded cohomology}
    The (bounded) cohomology group and of racks is defined similarly to that of groups.
    Since we will consider bounded cohomology and quasimorphisms, we restrict our attention to the trivial real coefficient $\RR$.
    
    For a rack $X$ and a non-negative integer $n$, let $C^n (X; \RR)$ be the set of functions $X^n \to \RR$.
    For $n < 0$, let $C^n (X; \RR) = 0$.
    Here we understand $X^0$ to be a one element set. 
    The coboundary operator $\delta^n \colon C^n (X; \RR) \to C^{n+1} (X; \RR)$ is defined by 
    \begin{equation}
        \begin{aligned}
            \delta^n f (x_1, \dots, x_n, x_{n+1}) &= \sum_{i=1}^{n+1} (-1)^i \left[
                f (x_1, \dots, x_{i-1}, x_{i+1}, \dots, x_{n+1}) \right. \\
            & \qquad \left.
                - f (x_1 \opr x_i, \dots, x_{i-1} \opr x_i, x_{i+1}, \dots, x_{n+1})
            \right]. 
        \end{aligned}
    \end{equation}
    In case of $n \leq 0$, we define $\delta^n = 0$.
    Thus, we obtain a cochain complex $C^\ast (X; \RR) = (C^n (X; \RR), \delta^n)$. 
    The \emph{rack cohomology} is the cohomology of this complex
    \begin{equation}
        H^n (X; \RR) = \ker \delta^n / \img \delta^{n-1}.
    \end{equation}

    \begin{example}
        We demonstrate some computations.
        \begin{equation}
            \begin{aligned}
                \delta^1 f (x, y) &= f (x) - f (x \opr y) \\
                \delta^2 f (x, y, z) &= f (x, z) - f (x, y) - f (x \opr y, z) + f (x \opr z, y \opr z) 
            \end{aligned}    
        \end{equation}
        Notice that $\delta^1 f = 0$ if and only if $f$ is constant on each connected component.
        Moreover, if $\RR$ is equipped with the structure of a trivial rack, then this is also equivalent to that $f$ is a rack homomorphism.
    \end{example}

    The bounded cohomology of a rack is defined similarly to that of a group.
    For a function $f \colon Z \to \RR$, we write $\norm{f}_\infty = \sup \set{\ord{f (z)} \mid z \in Z}$.
    For a rack $X$, let $C^n_\bdd (X; \RR)$ be the submodule of functions $X^n \to \RR$ which is bounded with respect to $\norm{\cdot}_\infty$.
    The coboundary operators $\delta^n$ may be restricted to $C^n_\bdd (X; \RR)$, and $C^\ast_\bdd (X; \RR) = (C^n_\bdd (X; \RR), \delta^n)$ forms a cochain complex.
    The \emph{bounded rack cohomology} $H^n_\bdd (X; \RR)$ is the cohomology of this complex.
    The inclusion $C^n_\bdd (X; \RR) \hookrightarrow C^n (X; \RR)$ induces the map $c^n \colon H^n_\bdd (X; \RR) \to H^n (X; \RR)$ called the \emph{comparison map}.

    If $X$ is a quandle, the \emph{quandle cochain complex} of $X$ is the quotient of the rack cochain complex of $X$ by the subcomplex $D^\ast (X; \RR)$ defined by 
    \begin{equation}
        D^n (X; \RR) = \set{
            f \in C^n (X; \RR) \mid f (x) = 0 \text{ for each } x \text{ such that } x_i = x_{i+1} \text{ for some } i
        }
    \end{equation}
    for $n \geq 2$ and $D^n (X; \RR) = 0$ for $n \leq 1$.
    The \emph{quandle cohomology} of $X$ is a cohomology of this complex and 
    the \emph{bounded quandle cohomology} is the bounded cohomology in this sense.

    \section{Rack quasimorphisms on the free product}\label{section:free-product}
    In this section, we present the method to construct quasimorphisms on the free product of racks using quasimorphisms on the free product of groups introduced by Rolli \cite{Rolli2009-wy}, and then we prove our main theorem using such rack quasimorphisms.
    
    \subsection{Group quasimorphisms}
    We first recall Rolli's construction and give some technical remarks on the elements in the free product of racks. 
    We will use those observations in our argument.

    A \emph{group quasimorphism} on a group $\Gamma$ is a function $\phi \colon \Gamma \to \RR$ satisfying 
    \begin{equation}
        D (\phi) = \sup_{g, h \in \Gamma} \ord{\phi (g) + \phi (h) - \phi (g h)} < \infty.
    \end{equation}
    The constant $D (\phi)$ is called the \emph{defect} of $\phi$. 
    In terms of group cohomology, $\phi$ is a quasimorphism if and only if the cocycle $\delta^1 \phi$ is bounded.

    Let $\Gamma = \bigast_{s \in S} \Gamma_s$ be the free product of (non-trivial) groups with $2 \leq \ord{S}$. 
    Each element $1 \neq g \in \Gamma$ is uniquely written in the product of $g_1, \dots, g_n$ with $1 \neq g_i \in \Gamma_{s_i}$ and $s_i \not= s_{i+1}$ for $1 \leq i < n$.
    Such product $g = g_1 \cdots g_n$ is called the \emph{factorization} of $g$, and we will call each $g_i$ the \emph{syllable} in the factorization. 

    Let $\lambda = (\lambda_s)_{s \in S}$ be a uniformly bounded family of odd bounded functions, that is, each function $\lambda_s \colon \Gamma_s \to \RR$ is bounded with respect to $\norm{\cdot}_\infty$ and satisfies $\lambda_s (g^{-1}) = - \lambda_s (g)$ for any $g \in \Gamma_s$, and $\sup \{ \norm{\lambda_s}_\infty \mid s \in S \} < \infty$.
    \begin{proposition}[\cite{Rolli2009-wy}, Proposition 4.1]\label{proposition:rolli-4-1}
        Let $\Gamma$ and $\lambda$ be as above. 
        Then a map $\phi_\lambda \colon \Gamma \to \RR$ defined by 
        \begin{equation}
            \phi_\lambda (g) = \sum_{i = 1}^{n} \lambda_{s_i} (g_i),
        \end{equation} 
        where $g = g_1 \cdots g_n$ is the factorization, is a group quasimorphism.
    \end{proposition}

    \subsection{Rack quasimorphisms}
    A \emph{rack quasimorphism} on a rack $X$ is a function $\phi \colon X \to \RR$ satisfying
    \begin{equation}
        \sup_{x, y \in X} \ord{\phi (x) - \phi (x \opr y)} < \infty.
    \end{equation}
    As in the case of group quasimorphisms, $\phi$ is a rack quasimorphism if and only if $\delta^1 \phi$ is bounded by the definition of the rack coboundary map.
    The following is the key observation that rack quasimorphisms give rise to bounded rack cohomology classes. 
    \begin{lemma}[5.2 in \cite{Kedra2024-jg}]
        If $\phi \colon X \to \RR$ is a rack quasimorphism, then $\delta^1 \phi \in C^2_\bdd (X; \RR)$ is a cocycle, and $[\delta^1 \phi] \in \ker c^2$.
        Moreover, if $\phi$ is unbounded on a connected component of $X$, then the cocycle $\delta^1 \phi$ is non-trivial.
    \end{lemma}

    Let $X = \bigast_{s \in S} X_s$ be the free product of racks $X_s$. 
    Observe that for each $t \in S$, if $x \in X_t$, $g \in \bigast_{s \in S} \adg (X_s)$ with the factorization $g = g_0 \cdot g_1 \cdots g_n $, and $g_0 \in \adg (X_t)$, then the element $(x, g)$ in $X$ can be written in the form 
    \begin{equation}
        (x \cdot g_0, g_1 \cdots g_n).
    \end{equation}
    Therefore, we can assume that each element in the free product $X$ has the form 
    \begin{equation}
        (x, g_1 \cdots g_n) \label{equation:reduced}
    \end{equation}
    where $x \in X_t$ and $g_1 \in \adg (X_{s_1})$ with $s_1 \not= t$.
    \begin{definition}
        We say that an element in a free product is \emph{reduced} if it has the form \eqref{equation:reduced}.
    \end{definition}

    Using a quasimorphism on the free product of the adjoint groups, we can construct a quasimorphism on the free product of racks as follows.
    \begin{proposition}\label{proposition:rack-quasimorphism}
        Let $X = \bigast_{s \in S} X_s$ be the free product of racks with $2 \leq \ord{S}$, $\lambda = (\lambda_s)_{s \in S}$ be a uniformly bounded family of odd bounded functions $\lambda_s \colon \adg (X_s) \to \RR$, and $\phi_\lambda$ be the group quasimorphism constructed in Proposition \ref{proposition:rolli-4-1}.
        Then a map $\hat{\phi}_\lambda \colon X \to \RR$ defined by 
        \begin{equation}
            \hat{\phi}_\lambda (x, g) = \phi_\lambda (g),
        \end{equation} 
        where $(x, g)$ is reduced, is a rack quasimorphism.
        Moreover, if $\lambda \neq 0$ then $\hat{\phi}_\lambda$ is unbounded.
    \end{proposition}
    \begin{proof}
        Assume that $(x, g), (y, h) \in X$ are reduced with the factorizations $g = g_1 \cdots g_m$ ($g_i \in \adg (X_{s_i})$) and $h = h_1 \cdots h_n$ ($h_j \in \adg (X_{t_j})$), and that $y \in X_{t_0}$ for certain $t_0$.

        By construction, we have 
        \begin{equation}
            \hat{\phi}_\lambda (x, g) - \hat{\phi}_\lambda (x, g h^{-1} e_y h)
            = 
            \phi_\lambda (g) - \phi_\lambda (g h^{-1} e_y h).
        \end{equation}
        Using the factorizations of $g$ and $h$, 
        \begin{equation}
            g h^{-1} e_y h
            = 
            (g_1 \cdots g_m) \cdot (h_n^{-1} \cdots h_1^{-1}) \cdot e_y \cdot (h_1 \cdots h_n)
        \end{equation}
        and this may not be the factorization of $g h^{-1} e_y h$.
        Notice that cancellation or concatenation do not occur in $h_1^{-1} e_y h_1$ since $t_0 \not= t_1$ by the assumption that $(y, h)$ is reduced.
        Let $r$ be the number of pairs of syllables which are cancelled between $g$ and $h^{-1} e_y h$.
        Cancellation will be the following cases.

        \noindent
        Case 1: 
        $1 \leq r < n$, that is, $g_m \cdot h_n^{-1} = 1, \dots, g_{m - r + 1} \cdot h_{n - r + 1}^{-1} = 1$.

        \noindent
        (i) If $s_{m - r} \neq t_{n - r}$, then no more cancellation occur, and the factorization is of the form
        \begin{equation}
            g_1 \cdots g_{m - r} \cdot h_{n - r}^{-1} \cdots h_1^{-1} \cdot e_y \cdot h_1 \cdots h_n.
        \end{equation}

        \noindent
        (ii) If $s_{m - r} = t_{n - r}$, let $u_r := s_{m - r} = t_{n - r}$ and $k_r := g_{m - r} \cdot h_{n - r}^{-1} \in \adg (X_{u_r})$.
        Then the factorization is of the form
        \begin{equation}
            g_1 \cdots g_{m - r - 1} \cdot k_r \cdot h_{n - r - 1}^{-1} \cdots h_1^{-1} \cdot e_y \cdot h_1 \cdots h_n,
        \end{equation}
        since $u_r \neq s_{m - r - 1}, t_{n - r - 1}$.
        
        \noindent
        Case 2: 
        $r = n$, that is, $g_m \cdot h_n^{-1} = 1, \dots, g_{m - n + 1} \cdot h_1^{-1} = 1$.

        \noindent
        (i) If $s_{m - n} \neq t_0$, then 
        \begin{equation}
            g_1 \cdots g_{m - n} \cdot e_y \cdot h_1 \cdots h_n.
        \end{equation}

        \noindent
        (ii) If $s_{m - n} = t_0$, let $u_n := s_{m - n} = t_0$ and $k_n := g_{m - n} \cdot e_y$. 
        Then 
        \begin{equation}
            g_1 \cdots g_{m - n - 1} \cdot k_n \cdot h_1 \cdots h_n.
        \end{equation}

        \noindent
        Case 3:
        $r > n$, that is, $g_m \cdot h_n^{-1} = 1, \dots, g_{m - n + 1} \cdot h_1^{-1}$ = 1, $g_{m - n} \cdot e_y = 1$, $g_{m - n + 1} \cdot h_1 = 1, \dots, g_{m - r + 1} \cdot h_{r - n - 1} = 1$.

        \noindent
        (i) If $s_{m - r} \neq t_{r - n}$, then 
        \begin{equation}
            g_1 \cdots g_{m - r} \cdot h_{r - n} \cdots h_n.
        \end{equation}

        \noindent
        (ii) If $s_{m - r} = t_{r - n}$, let $u_r := s_{m - r} = t_{r - n}$ and $k_r := g_{m - r} \cdot h_{r - n}$. 
        Then 
        \begin{equation}
            g_1 \cdots g_{m - r - 1} \cdot k_r \cdot h_{r - n + 1} \cdots h_n.
        \end{equation}

        Case 1-(ii) is most critical for our estimation. 
        Since $\lambda_s$'s are odd functions, $\lambda_{s_m} (g_m) = \lambda_{t_n} (h_n)$, $\dots$, $\lambda_{s_{m - r + 1}} (g_{m - r + 1}) = \lambda_{t_{n - r + 1}} (h_{n - r + 1})$ in this case, and we have 
        \begin{equation}
            \begin{aligned}
                &\ord{\hat{\phi}_\lambda (x, g) - \hat{\phi}_\lambda (x, g h^{-1} e_y h)} \\
                \leq \,
                &\left\lvert
                    \sum_{i = 1}^m \lambda_{s_i} (g_i) -
                    \left(
                        \sum_{i = 1}^{m - r - 1} \lambda_{s_i} (g_i)
                        + \lambda_{u_r} (k_r)
                        - \sum_{i = 1}^{n - r - 1} \lambda_{t_i} (h_i)
                        + \lambda_{t_0} (e_y)
                        + \sum_{i = 1}^n \lambda_{t_i} (h_i)
                    \right)
                \right\rvert \\
                \leq \,
                &\ord{\lambda_{s_{m-r}} (g_{m-r})} + \ord{\lambda_u (k)} + \ord{\lambda_{t_0} (e_y)} + \ord{\lambda_{t_{n-r}} (h_{n-r})}.
            \end{aligned}
        \end{equation}
        Thus, in general, 
        \begin{equation}
            \begin{aligned}
                \ord{\hat{\phi}_\lambda (x, g) - \hat{\phi}_\lambda (x, g h^{-1} e_y h)}
                \leq 
                4 \cdot \norm{\lambda}_\infty 
                <
                \infty.
            \end{aligned}
        \end{equation}
        Therefore, $\hat{\phi}_\lambda$ is a rack quasimorphism.

        Next, assume $\lambda \neq 0$. 
        There is $s_0 \in S$ and $g_0 \in \adg (X_{s_0})$ such that $\lambda_{s_0} \neq 0$ and $\lambda_{s_0} (g_0) \neq 0$.
        Take $t \in S$ and $x \in \adg (X_t)$ for $t \neq s_0$.
        Since $\lambda_t$ is odd, at least one of $\lambda_{s_0} (g_0) + \lambda_t (e_x^\pm) \neq 0$, and we may assume $\lambda_{s_0} (g_0) + \lambda_t (e_x) \neq 0$ without loss of generality. 
        For any $n > 0$, $(x, (g_0 e_x)^n)$ is reduced.
        Thus, we have 
        \begin{equation}
            \ord{\hat{\phi}_\lambda (x, (g_0 e_x)^n)}
            = \ord{\phi_\lambda ((g_0 e_x)^n)}
            = n \cdot \ord{\lambda_{s_0} (g_0) + \lambda_t (e_x)}
            > 0.
        \end{equation}
        Therefore, $\hat{\phi}_\lambda$ is unbounded.
    \end{proof}
    
    \subsection{Proof of Theorem \ref{theorem:main}}
    To prove Theorem \ref{theorem:main}, we will show the existence of an injective linear map from an infinite-dimensional vector space.
    Quasimorphisms constructed above provide an injective linear map from the space of uniformly bounded families of odd bounded functions to the second bounded cohomology.
    For the free product of nontrivial groups $\Gamma = \bigast_{s \in S} \Gamma_s$, we write the space of uniformly bounded families of odd bounded functions by $V_0 (\Gamma)$:
    \begin{equation}
        V_0 (\Gamma) = \set{ \lambda = (\lambda_s)_{s \in S} \mid \text{each } \lambda_s \colon \Gamma_s \to \RR \text{ is odd and bounded, and } \sup_{s \in S} \norm{ \lambda_s }_\infty < \infty }.
    \end{equation}

    \begin{proposition}\label{proposition:injective-linear}
        Let $X = \bigast_{s \in S} X_s$ be the free product of finitely generated racks with $2 \leq \ord{S} < \infty$.
        Then a map $\hat{\phi} \colon V_0 (\adg (X)) \to H_\bdd^2 (X, \RR)$ defined by 
        \begin{equation}
            \hat{\phi} (\lambda) = [ \delta^1 \hat{\phi}_\lambda ]_\bdd 
        \end{equation}
        is an injective linear map.
    \end{proposition}
    \begin{proof}
        Linearity is clear by construction.

        To show injectivity, suppose $[ \delta^1 \hat{\phi}_\lambda ]_\bdd = 0$.
        There is a bounded function $\beta \in C_\bdd^1 (X, \RR)$ such that $\delta^1 \beta = \delta^1 \hat{\phi}_\lambda$.
        Let $f = \hat{\phi}_\lambda - \beta$. 
        Since $f$ is a rack homomorphism and $\RR$ is considered to be a trivial rack, $f$ is constant on each connected component of $X$.
        Since $\ord{S} < \infty$ and each $X_s$ is finitely generated, the number of connected components of $X$ is finite.
        Therefore, $f$ is bounded, and then $\hat{\phi}_\lambda$ is bounded. 
        By the latter half of Proposition \ref{proposition:rack-quasimorphism}, we can conclude that $\lambda = 0$.
    \end{proof}

    The space $V_0 (\adg (X))$ turns out to be infinite-dimensional by the nature of adjoint groups.
    \begin{proof}[Proof of Theorem \ref{theorem:main}]
        By Proposition \ref{proposition:injective-linear}, it is enough to show that the space $V_0 (\adg (X))$ is infinite-dimensional.
        
        Let $\ell^\infty_\mathrm{odd} (\ZZ)$ be the space of odd bounded functions $\ZZ \to \RR$.

        The subgroup of the adjoint group generated by its generator $e_x$ is infinite cyclic.
        Fix $s_0 \in S$ and $x_0 \in X_{s_0}$.
        Define a map $\iota \colon \ell^\infty_\mathrm{odd} (\ZZ) \to V_0 (\adg (X))$ by 
        \begin{equation}
            \iota (\sigma) = (\iota (\sigma)_s)_{s \in S}, \quad 
            \iota (\sigma)_s (g) = 
            \begin{cases}
                \sigma (n) & \text{for } s = s_0 \text{ and } g = e_{x_0}^n \\ 
                0 & \text{otherwise}
            \end{cases}.
        \end{equation}
        By construction, each function $\iota (\sigma)_s$ is odd and bounded, and a family $\iota (\sigma)$ is uniformly bounded.
        Thus, $\iota$ is well-defined.
        Clearly it is linear and injective.
        Therefore, the space $V_0 (\adg (X))$ is infinite-dimensional.
    \end{proof}
    \begin{remark}
        Proposition \ref{proposition:injective-linear} also holds for quandles. 
        More precisely, if $X_s$'s are quandles in addition, then $\delta^1 \hat{\phi}_\lambda$ is a quandle cocycle.
        Indeed, $\delta^1 \hat{\phi}_\lambda$ is a rack cocycle by the above argument, and 
        \begin{equation}
            \delta^1 \hat{\phi}_\lambda ((x, g), (x, g))
            =
            \phi_\lambda (g) - \phi_\lambda (g)
            =
            0
        \end{equation}
        by the operation on $\bigast_{s \in S} X_s$. 
    \end{remark}
    \begin{remark}[cf. \cite{Rolli2009-wy}, Corollary 4.3]
        For the free product $\Gamma$ of groups in general, the space $V_0 (\Gamma)$ is not necessarily infinite-dimensional.
        For example, when $\Gamma = \mathrm{PSL}_2 (\ZZ) \cong \ZZ / {2} \ast \ZZ / {3}$, $\dim V_0 (\Gamma) = 1$.
    \end{remark}

    \section{Homogeneous group quasimorphisms and free racks}\label{section:homogeneous}
    In this section, we give another proof of Corollary \ref{corollary:main} using homogeneous group quasimorphisms.

    A group quasimorphism $\phi \colon \Gamma \to \RR$ is \emph{homogeneous} if $\phi (g^n) = n \cdot \phi (g)$ for any $g \in \Gamma$ and $n \in \ZZ$.
    We write the space of homogeneous quasimorphisms on $\Gamma$ by $Q^\homog (\Gamma)$.
    It is known that a homogeneous quasimorphism is constant on conjugacy classes, that is, $\phi (h^{-1} g h) = \phi (g)$ for any $g, h \in \Gamma$.
    Analogous statements of Proposition \ref{proposition:rack-quasimorphism} and \ref{proposition:injective-linear} hold with slight modification of arguments.

    We can construct a quandle quasimorphism on the free quandle on two generators using a homogeneous group quasimorphism on the free group on two generators (cf. Example 3.15 in \cite{Kedra2024-jg}).
    This construction is generalized to the free product of racks as follows.
    \begin{proposition}\label{proposition:rack-quasimorphism-homogeneous}
        Let $X = \bigast_{s \in S} X_s$ be the free product of racks with $2 \leq \ord{S} < \infty$. 
        Let $\phi \colon \adg (X) \to \RR$ a homogeneous quasimorphism.
        If each $X_s$ is finitely generated and $\ord{S} < \infty$, then a map $\hat{\phi} \colon X \to \RR$ defined by 
        \begin{equation}
            \hat{\phi} (x, g) = \phi (g), \label{equation:using_homogeneous}
        \end{equation}
        where $(x, g)$ is reduced, is a rack quasimorphism.
    \end{proposition}
    \begin{proof}
        Since $\phi$ is a homogeneous quasimorphism, we have 
        \begin{equation}
            \begin{aligned}
                &\quad
                \ord{\hat{\phi} (x, g) - \hat{\phi} (x, g h^{-1} e_y h)} \\
                &\leq
                \ord{\phi (g) - \phi (g h^{-1} e_y h) + \phi (h^{-1} e_y h)} + \ord{\phi (h^{-1} e_y h)} \\
                &\leq 
                D (\phi) + \ord{\phi (e_y)}.
            \end{aligned}
        \end{equation}

        Since each $X_s$ is finitely generated and $\ord{S} < \infty$, there is a constant $M \geq 0$ such that $\ord{\phi (e_y)} \leq M$ for any $y \in \sqcup_{s \in S} X_s$.
        Therefore, $\hat{\phi}$ is a rack quasimorphism. 
    \end{proof}
    \begin{proposition}\label{proposition:injective-linear-homogeneous}
        Let $X = \bigast_{s \in S} X_s$ be the free product of finitely generated racks with $2 \leq \ord{S} < \infty$, and $\phi \colon \adg (X) \to \RR$ a quasimorphism on the adjoint group of $X$.
        Then a map $\eta \colon Q^\homog (\adg (X)) \to H_\bdd^2 (X, \RR)$ defined by 
        \begin{equation}
            \eta (\phi) = [\delta^1 \hat{\phi}]_\bdd 
        \end{equation}
        is an injective linear map.
    \end{proposition}
    Almost same argument in the proof of Proposition \ref{proposition:injective-linear} also works in this case. 
    \begin{proof}
        Suppose $[\delta^1 \hat{\phi}]_\bdd = 0$ then $\hat{\phi}$ is bounded by the same argument in Proposition \ref{proposition:injective-linear}.
        Let $g \in \adg (X)$ with the factorization $g = g_1 \cdots g_m$ and $n \in \ZZ$. 
        When $g_1 \in \adg (X_{s_1})$, we can choose $s_0 \not= s_1$ and fix arbitrary $x \in X_{s_0}$.
        Then we have 
        \begin{equation}
            \hat{\phi} (x, g^n) = \phi (g^n) = n \cdot \phi (g).
        \end{equation}
        Since $\hat{\phi}$ is bounded, $\phi = 0$.
    \end{proof}

    By combining the following results, we can conclude that the space $Q^\homog (\Gamma)$ is infinite-dimensional.
    The kernel of the comparison map $c^n \colon H^n_\bdd (\Gamma, \RR) \to H^n (\Gamma, \RR)$ is called the \emph{exact bounded cohomology}, $EH^n_\bdd (\Gamma, \RR)$.
    This is the subspace consisting of the classes which represented by a bounded cocycle.
    The following relation between the space of homogeneous quasimorphisms and the second exact bounded cohomology is well-known (see for example \cite{Frigerio2017-kr}):
    \begin{lemma}
        $Q^\homog (\Gamma) / {\hom (\Gamma, \RR)} \cong EH^2_\bdd (\Gamma, \RR)$.
    \end{lemma}
    Recall that the adjoint group of the free rack is the free group and that the free group of rank $\geq 2$ is non-elementary word-hyperbolic.
    \begin{proposition}[cf. \cite{Epstein1997-od}, Theorem 3.1]
        If a group $\Gamma$ is non-elementary word-hyperbolic, then the exact second bounded cohomology $EH_\bdd^2 (\Gamma, \RR)$ is infinite-dimensional.
    \end{proposition}
    Thus, $Q^\homog (\adg X)$ is infinite-dimensional, and we obtain another proof of Corollary \ref{corollary:main}.
    Note that the free product of groups is not necessarily hyperbolic, but relatively hyperbolic.

    \bibliographystyle{plain}
    \bibliography{references}

\end{document}